\documentclass[11pt,reqno]{amsart}
\usepackage[T1]{fontenc}
\usepackage{tikz}
\usetikzlibrary{decorations.pathreplacing}
\usepackage{amsmath,amssymb,amsthm,mathtools,enumerate}
\usepackage[margin=1.25in]{geometry}
\usepackage{hyperref}
\allowdisplaybreaks
\newtheorem{theorem}{Theorem}[section]

\newtheorem{lemma}[theorem]{Lemma}
\newtheorem{proposition}[theorem]{Proposition}
\theoremstyle{remark}
\newtheorem{remark}[theorem]{Remark}
\numberwithin{equation}{section}

\newcommand{\R}{\mathbb{R}}

\newcommand{\ut}{u^\theta}
\newcommand{\ur}{u^r}
\newcommand{\uz}{u^z}
\newcommand{\eps}{\varepsilon}
\newcommand{\at}{\lambda}

\title[A logarithmic regularity criterion for the axisymmetric NSE]{A logarithmic regularity criterion for\\ the axisymmetric Navier--Stokes equations}
\author{Stan Palasek}
\address{Department of Mathematics, Princeton University, Princeton, NJ 08540, USA}
\email{spalasek@princeton.edu} 

\begin{document}

\begin{abstract}
For axisymmetric solutions of the Navier--Stokes equations in $\R^3$, the circulation $\Gamma=r\ut$ obeys a maximum principle, and the resulting a~priori bound on $\Gamma$ is scale-invariant. Nonetheless, the global regularity problem remains completely open. The sharpest known regularity criteria in terms of the circulation are slightly subcritical, namely $|\Gamma|\lesssim(\log r^{-1})^{-k}$ near the symmetry axis; Lei and Zhang proved regularity for $k=2$, and Wei later lowered the exponent to $k=3/2$. In this note, we strengthen the criterion further to $k=5/4$. The argument follows the usual route of energy estimates for the quantities $\omega^r/r$ and $\omega^\theta/r$, but now with a sharper bound on the vortex stretching term in the equation for  $\omega^r/r$, in which the loss is proportional only to the square root of the number of scales between the inner vortex core and the far field.
\end{abstract}

\maketitle

\section{Introduction}

We consider axisymmetric solutions of the incompressible Navier--Stokes equations in $\R^3$ with unit viscosity,
\begin{align}\label{eq:nse}
\partial_tu-\Delta u+u\cdot\nabla u+\nabla p=0,\qquad\operatorname{div}u=0.
\end{align}
For an axisymmetric velocity field $u=\ur e_r+\ut e_\theta+\uz e_z$, the system \eqref{eq:nse} can be written in terms of the circulation $\Gamma\coloneqq r\ut$ and the rescaled azimuthal vorticity $\Omega\coloneqq\omega^\theta/r$ as
\begin{align}\label{eq:ns}
\begin{aligned}
&\partial_t\Gamma+b\cdot\nabla\Gamma-\Big(\Delta-\frac2r\partial_r\Big)\Gamma=0,\\
&\partial_t\Omega+b\cdot\nabla\Omega-\Big(\Delta+\frac2r\partial_r\Big)\Omega=\frac{\partial_z(\Gamma^2)}{r^4},
\end{aligned}
\end{align}
where the swirl-free part of the velocity $b\coloneqq\ur e_r+\uz e_z$ is recovered from $\omega^\theta=r\Omega$ through the stream function $\psi$:
\begin{align*}
-\Big(\Delta-\frac1{r^2}\Big)\psi=\omega^\theta,\qquad \ur=-\partial_z\psi,\qquad \uz=\frac1r\partial_r(r\psi).
\end{align*}
In the absence of swirl, the system reduces to a drift-diffusion equation for $\omega^\theta/r$, and global regularity is classical \cite{Lad68,UY68,LMNP99}. With swirl, the equation for the vorticity $\omega^\theta$ contains the additional source term $\partial_z(\ut)^2/r$, and the global regularity problem remains open. The question of uniqueness of Leray--Hopf solutions of \eqref{eq:nse} in the axisymmetric class has recently been resolved in the negative by Hou, Wang, and Yang \cite{hou2025nonuniqueness}, and likely extends even to the swirl-free case \cite{ionescu2026non}. 

The axisymmetric system possesses significantly more structure than the general Navier--Stokes equations, including an additional a~priori bound at the critical regularity. Indeed, by its equation in \eqref{eq:ns}, the circulation obeys the maximum principle \cite{CL02,CSTY08}
\begin{align*}
    \|\Gamma(t)\|_{L^\infty}\leq\|\Gamma_0\|_{L^\infty}.
\end{align*}Unlike the energy, which is supercritical in dimension three, this bound is critical, i.e., invariant under the scaling $u_\lambda(x,t)=\lambda u(\lambda x,\lambda^2t)$ of \eqref{eq:nse}. Thus, a natural route toward global regularity is to understand the extent to which the swirl controls the full Navier--Stokes solution, and to pursue conditional regularity theorems under stronger hypotheses on $\Gamma$.

\subsection{Known axisymmetric criteria}
While it is not known whether the a~priori bound for $\Gamma$ suffices in general to control the regularity of the solution, positive results have been obtained with slightly strengthened conditions. In the remarkable work \cite{CFZ17}, Chen, Fang, and Zhang showed that $|\Gamma|\leq Cr^\alpha$ for some $\alpha>0$ implies regularity, as a consequence of the Prodi--Serrin--Ladyzhenskaya-type criterion
\begin{align*}
r^d\ut\in L^p_tL^q_x,\qquad\frac2p+\frac3q\leq1-d,\qquad0\leq d<1.
\end{align*}
This falls short of a global regularity proof only marginally, corresponding to the missing endpoint $(d,p,q)=(1,\infty,\infty)$. To try to close this gap, it is natural to consider slightly subcritical bounds on the circulation of the form
\begin{align}\label{eq:logk}
|\Gamma(t,r,z)|\lesssim \Big(\log \frac1r\Big)^{-k},
\end{align}
with the a~priori bound corresponding to the case $k=0$. Note that some results differ in what size is allowed for the implicit constant, but we neglect these differences for simplicity of the discussion.

Q.~S.~Zhang in \cite{Z15} first proved regularity under the assumption \eqref{eq:logk} with any $k>2$. Lei and Zhang \cite{LZ17} then reached $k=2$ for data in $H^{1/2}$, as a consequence of their Form Boundedness Condition on the swirl. Wei \cite{W16} further lowered the exponent to $k=\frac32$ from any data $u_0\in H^2$ with $\Gamma_0\in L^\infty$. The case of the criterion \eqref{eq:logk} proved in \cite{W16} follows from a more quantitative statement, analogous to Theorem~\ref{thm:quant} below. Up to this point, Wei's exponent has been the state of the art, and the question of whether it can be improved has been asked, for instance, in \S2.2 of \cite{Z24}.

There are many other types of axisymmetric regularity criteria, and we do not attempt an exhaustive review here. The original works ruling out certain types of Type~I blow-up are due to Koch--Nadirashvili--Seregin--\v Sver\'ak \cite{MR2545826} and Chen--Strain--Tsai--Yau \cite{CSTY08,MR2512859} who showed, among other things, regularity for solutions obeying $|u|\leq C/r$. By the maximum principle for $\Gamma$, it suffices to assume a bound on just the swirl-free part $b$, for which Pan \cite{Pan16} showed the slightly supercritical refinement
\begin{align*}
r|b|\leq\Big(\log\log\frac 3r\Big)^\alpha
\end{align*}
sufficiently close to the axis, where $\alpha>0$ is a small absolute constant; see also \cite{Ser22}. Later, Li and Pan \cite{LP19} proved regularity under the criterion
\begin{align*}
    r|b|\lesssim\Big(1+\log \frac1r\Big)^\beta,\quad|\Gamma|\lesssim\Big(1+\log \frac1r\Big)^{-\alpha}
\end{align*}
with $0\leq\beta<\alpha/6$, which allows every $k>0$ in the swirl bound \eqref{eq:logk} at the expense of having to assume a bound on $b$. There is an intriguing one-sided criterion due to Q.~S.~Zhang \cite{QZ26}, namely that $\ur\geq-C/\sqrt{T-t}$ implies regularity for any $C>0$. There have also been a number of quantitative regularity results in the axisymmetric setting, e.g., \cite{Pal21,CTZ22,OP23,LR24}.

\subsection{Main theorems}\label{sec:main}

Our main regularity criterion can be stated as follows.

\begin{theorem}\label{thm:main}
There is an absolute constant $c_*>0$ with the following property. Let $u_0\in H^2(\R^3)$ be axisymmetric and divergence-free with $\Gamma_0\in L^\infty$, let $u\in C([0,T_*);H^2)$ be the strong solution with maximal existence time $T_*$, and let $\delta_0\in(0,\frac12)$. If
\begin{align*}
|\Gamma(t,r,z)|\leq c_*\Big(\log \frac1r\Big)^{-5/4}
\end{align*}
for all $0<r\leq\delta_0$, $z\in\R$, and $t\in(0,T_*)$, then $T_*=\infty$.
\end{theorem}

As in \cite{W16}, we derive the main theorem from a quantitative statement in which the smallness of the circulation near the axis is measured against the size of the data. First, we define the rescaled vorticities
\begin{align*}
\Omega&\coloneqq \frac{\omega^\theta}r,\qquad J\coloneqq \frac{\omega^r}r=-\frac{\partial_z\ut}r
\end{align*}
as well as the quantities
\begin{align*}
M&\coloneqq \|\Gamma_0\|_{L^\infty},\qquad M_1\coloneqq (1+M)\|u_0\|_2,\qquad M_0\coloneqq (\|J_0\|_2+\|\Omega_0\|_2)M_1^3
\end{align*}
where we take norms with respect to the measure $dx\coloneqq r\,dr\,dz$, so that
\begin{align*}
\|g\|_2^2\coloneqq\int|g|^2r\,dr\,dz
\end{align*}
differs from the typical $L^2(\R^3)$ norm by a factor of $(2\pi)^{-1}$. We point out that the importance of the pair $(\Omega,J)$ appears to have been discovered in \cite{CFZ17} and \cite{LZ17}. We assume throughout that $u_0\neq0$, which guarantees that $M_0,M_1>0$.

\begin{theorem}\label{thm:quant}
There exist absolute constants $c_0,C_0>0$ such that we have the following. Let $u$ be as in Theorem~\ref{thm:main}, and suppose that for some $r_0,\eps,S\ge0$ obeying
\begin{align}\label{ass1}
    \eps^4(1+S)^5\leq c_0
\end{align}
and
\begin{align}\label{ass2}
    S\geq\log\Big(C_0\,\eps^{-5}\max\big\{M_0^{1/4},\,r_0^{-1/2}M_1\big\}\Big),
\end{align}
we have
\begin{align}\label{ass3}
|\Gamma|\leq\eps\ \text{ on }\{r\leq r_0\}\times(0,T_*).
\end{align}
Then $T_*=\infty$.
\end{theorem}

Here, the parameter $S$ represents the number of logarithmic scales over which the argument can tolerate a vortex of circulation $\eps$ near the axis (see \S\ref{sec:strategy}). For $\eps\leq c_0^{1/4}$, the largest admissible value of $S$ is
\begin{align*}
S_*(\eps)\coloneqq c_0^{1/5}\eps^{-4/5}-1 .
\end{align*}
Writing $X\coloneqq \max\{M_0^{1/4},r_0^{-1/2}M_1\}$, we see that the hypotheses \eqref{ass1}--\eqref{ass3} assert that $|\Gamma|\leq\eps$ near the axis and
\begin{align*}
S_*(\eps)\geq\log(C_0\eps^{-5}X).
\end{align*}
For large $X$, this amounts to $\eps\lesssim(\log X)^{-5/4}$. To prove Theorem~\ref{thm:main}, we let $r_0\to0$ and take $\eps=c_*(\log r_0^{-1})^{-5/4}$. The lower bound on $S$ then grows like $\frac12\log r_0^{-1}$, while $S_*(\eps)=\log r_0^{-1}-1$; see \S\ref{sec:proofs}.

\begin{remark}\label{rem:data}
As in Remark~1.1 of \cite{W16}, the maximum principle gives a condition on the data alone. If
\begin{align*}
\|\Gamma_0\|_{L^\infty}\leq\eps\leq c_0^{1/4}\qquad\text{and}\qquad S_*(\eps)\geq\log(C_0\eps^{-5}M_0^{1/4}),
\end{align*}
then $T_*=\infty$. Indeed, \eqref{ass1}--\eqref{ass3} then hold for every large $r_0$, with $S=S_*(\eps)$. For large $M_0$, this improves Wei's condition
\begin{align*}
\|\Gamma_0\|_{L^\infty}\leq\big(1+\log(C_0M_0^{1/4}+1)\big)^{-3/2}.
\end{align*}
In other words, regularity follows from a smallness assumption on the swirl with a scale invariant threshold involving both subcritial and supercritical norms of the data. Whether such a criterion exists involving only critical norms of the data is an important open problem, although some interesting progress was made in \cite{MR4367867}.
\end{remark}

\subsection{Idea of the proof}\label{sec:strategy}
We consider the quantities $J=\omega^r/r$ and $\Omega=\omega^\theta/r$ which are smooth near the axis before the blow-up time and satisfy
\begin{align}\label{eq:OJintro}
\begin{aligned}
\partial_t\Omega+b\cdot\nabla\Omega-\Big(\Delta+\frac2r\partial_r\Big)\Omega&=-2\frac{\ut}rJ\\
\partial_tJ+b\cdot\nabla J-\Big(\Delta+\frac2r\partial_r\Big) J&=(\omega^r\partial_r+\omega^z\partial_z)\frac{\ur}r.
\end{aligned}
\end{align}
By, for instance, the continuation criterion used in \cite{CFZ17}, it suffices to bound $\|\Omega(t)\|_2$ uniformly up until $T_*$. Following Wei \cite{W16}, we work with a weighted energy
\begin{align*}
A(t)\coloneqq \|J\|_2^2+\mu\|\Omega\|_2^2
\end{align*}
for a constant $\mu>0$ to be chosen. In addition to the favorable terms $\|\nabla J\|_2^2+\mu\|\nabla\Omega\|_2^2$ from the dissipation, the energy identities contain the interactions
\begin{align}\label{T1T2}
T_1\coloneqq -4\int\frac{\ut}rJ\Omega\,dx\qquad\text{and}\qquad T_2\coloneqq \int J(\omega^r\partial_r+\omega^z\partial_z)\frac{\ur}r\,dx ,
\end{align}
coming from the equations for $\Omega$ and $J$, respectively.

Let us first consider $T_1$ (see Figure~\ref{fig:scales}). By the Cauchy--Schwarz inequality, it is controlled by $J$ and $\Omega$ measured in $L^2$ with the weight $|\ut|/r=|\Gamma|/r^2$. Since $J$ and $\Omega$ need not vanish on the axis, Hardy's inequality with weight $r^{-2}$ against the measure $r\,dr$ is not available. Wei's idea, which we adapt in the sequel, is to use a Hardy-type inequality truncated at a radius $r_1$, chosen at each time, that exceeds the radius of the core by a factor $K$. Here the \emph{core} is a carefully chosen region $\{r\lesssim\eps/\at(t)\}$, where $\at(t)$ is a velocity scale near the axis defined in \S\ref{sec:prelim}. Within it, circulation is small compared to $\eps$ on average (see Lemma~\ref{rad_lemma}). Between this core and $r_1$, the swirl may behave like a vortex of the form $\eps/r$ over $S=\log K$ logarithmic scales, and the resulting Hardy constant is $L\sim\eps S^2$. Outside $r_1$, one uses only the maximum principle $|\Gamma|\leq M$, and the resulting far-field terms can be controlled by the energy of $u$ as long as $S$ obeys the lower bound \eqref{ass2}, which forces $S\gtrsim\frac12\log r_0^{-1}$ as $r_0\to0$.

In general, with the Hardy constant $L$, the part of $T_1$ involving $J$ can be absorbed by $\|\nabla J\|_2^2$ in an energy estimate for $A(t)$ only if $\mu\lesssim L^{-2}$. On the other hand, if $|T_2|\leq \Sigma \|\nabla J\|_2\|\nabla\Omega\|_2$ for some coefficient $\Sigma$, then absorbing $T_2$ requires $\Sigma ^2\lesssim\mu$. Both requirements can be met when
\begin{align}\label{eq:window}
\Sigma ^2\lesssim\mu\lesssim L^{-2},
\end{align}
that is, when $\Sigma L\lesssim1$. Wei bounds $T_2$ using the same weighted Hardy inequality, which gives $\Sigma \sim\eps S$, so that \eqref{eq:window} becomes
$\eps^2S^3\lesssim1$. Since $S\sim\frac12\log r_0^{-1}$, this leads to $k=\frac32$. The key improvement leading to Theorem~\ref{thm:main} is that this bound can be strengthened to $\Sigma \sim\eps\sqrt S$, which relaxes the condition to $\eps^4S^5\lesssim1$ (cf.~\eqref{ass1}), thus allowing $k=\frac54$. (See again Figure~\ref{fig:scales}.)

To estimate $T_2$, we write it as
\begin{align*}
T_2=\int\ut\Big(\partial_zJ\,\partial_r\frac{\ur}r-\partial_rJ\,\partial_z\frac{\ur}r\Big)\,dx
\end{align*}
(Lemma~\ref{lem:energy}) and treat the two terms differently. In the first term, we write
\begin{align*}
\ut\,\partial_r\frac{\ur}r=\Gamma\,\frac1r\partial_r\frac{\ur}r,
\end{align*}
and observe that $r^{-1}\partial_r(\ur/r)$ is precisely a component of the Hessian $\nabla^2(\ur/r)$, so this term can be bounded by $\eps\|\partial_zJ\|_2\|\partial_z\Omega\|_2$ without the need for any Hardy inequality. In the second term, instead of placing the weight $|\ut|$ on $\partial_z(\ur/r)$ and using a Hardy inequality in $r$, we integrate in $z$ first. The trace inequality (Lemma~\ref{lem:trace}) bounds the $L^2_z$ norms of $\partial_z(\ur/r)$ along the cylinders $\{r=\text{const}\}$ by $\|\partial_z\Omega\|_2$, uniformly in $r$. Control of the circulation is used only to estimate
\begin{align*}
\int_0^{r_1}\sup_z\Gamma^2\,\frac{dr}r,
\end{align*}
which, we will show, is of size $\eps^2S$. Note that the naive swirl bound $\sup_z|\Gamma|\leq\eps$ alone would make this integral diverge at the axis, so the inner region is handled separately by a slightly more delicate estimate in Lemma~\ref{rad_lemma}.

\begin{figure}[t]
\centering
\pgfmathdeclarefunction{sig}{1}{\pgfmathparse{1/(1+exp(min(9,-#1)))}}
\pgfmathdeclarefunction{Wfun}{1}{\pgfmathparse{0.9*sig(3*(#1+0.5))+0.06*sig(2*(#1-1))-0.06*sig(2*(#1-2.5))+2*sig(3*(#1-6.6))}}
\begin{tikzpicture}[x=1.1cm,y=0.6cm,>=stealth,font=\small,
  paneltag/.style={anchor=north west,font=\footnotesize\scshape,fill=white,
    inner sep=2pt,rounded corners=1pt}]
\def\S{4} 
\def\Sh{3.31}
\def\Rz{5.3} 
\def\Hw{4.0} 
\def\Hh{5.2} 
\def\Off{4.8}
\begin{scope}[shift={(0,\Off)}]
\fill[black!7] (\Sh,0) rectangle (7.6,\Hh);
\fill[blue!20] (-3,0) -- plot[domain=-3:\S,samples=100] (\x,{Wfun(\x)*(\S-\x)}) -- cycle;
\foreach \p in {0,\S,\Rz} \draw[black!35,densely dotted] (\p,0) -- (\p,\Hh);
\draw (-3.2,0) -- (7.8,0);
\foreach \p in {0,\Sh,\S,\Rz} \draw (\p,0.08) -- (\p,-0.08);
\draw (-3.2,0) -- (-3.2,\Hh);
\draw (-3.25,1) -- (-3.15,1) node[pos=0,left] {$\eps$};
\draw (-3.25,\S) -- (-3.15,\S) node[pos=0,left] {$\eps S$};
\draw[thick] plot[domain=-3:\S,samples=100] (\x,{Wfun(\x)*(\S-\x)});
\node[anchor=south west] at (0.75,2.75) {$W\log\dfrac{r_1}r$};
\node at (1.35,0.95) {area $\lesssim\eps S^2$};
\node[paneltag] at (-3.05,\Hh-0.1) {Hardy inequality (Lemma~\ref{lem:hardy})};
\end{scope}
\fill[black!7] (\Sh,0) rectangle (7.6,\Hw);
\fill[orange!40] (-3,0) -- plot[domain=-3:\S,samples=80] (\x,{Wfun(\x)}) -- (\S,0) -- cycle;
\foreach \p in {0,\S,\Rz} \draw[black!35,densely dotted] (\p,0) -- (\p,\Hw);
\draw[->] (-3.2,0) -- (7.8,0) node[below] {$\log r$};
\draw (-3.2,0) -- (-3.2,\Hw);
\draw (-3.25,1) -- (-3.15,1) node[pos=0,left] {$\eps$};
\draw (-3.25,3.1) -- (-3.15,3.1) node[pos=0,left] {$M$};

\draw[dashed] (-3.2,1) -- (\Rz,1);
\draw[dashed] (-3.2,3.1) -- (7.6,3.1);

\draw[thick] plot[domain=-3:7.5,samples=200] (\x,{Wfun(\x)}) node[right] {$W$};
\node[font=\footnotesize] at (1.75,0.45) {area $\lesssim\eps S$};

\node at (-1.6,2.05) {core};
\node at (5.6,2.05) {far field};
\node[paneltag] at (-3.05,\Hw-0.1) {Trace inequality (Lemmas \ref{lem:trace} and \ref{lem:T2})};
\foreach \p/\l in {0/{\eps/\at},\Sh/{r_1/2},\S/{r_1},\Rz/{r_0}}
  \draw (\p,0.08) -- (\p,-0.08) node[below] {$\l$};
\draw[decorate,decoration={brace,mirror,amplitude=5pt}] (0,-1.55) -- (\S,-1.55)
  node[midway,below=5pt] {$S=\log K$ scales};
\end{tikzpicture}
\caption{The swirl $W(r)\coloneqq\sup_z|\Gamma(t,r,z)|$ as it appears in the two near-axis inequalities. In the worst case, $W\approx\eps$ across the $S=\log K$ scales between the inner core and $r_1$. The Hardy inequality (top) recovers $J$ and $\Omega$ at radius $r$ by integrating out to $r_1$, so the swirl at $r$ is counted once for each of the $\log(r_1/r)$ scales in between, and the area is of size $\eps S^2$. The trace inequality (bottom) integrates only in $z$, so each scale is counted once and the area is of size $\eps S$. This new scheme for estimating $T_2$ is the main source of improvement compared to \cite{W16}.}
\label{fig:scales}
\end{figure}
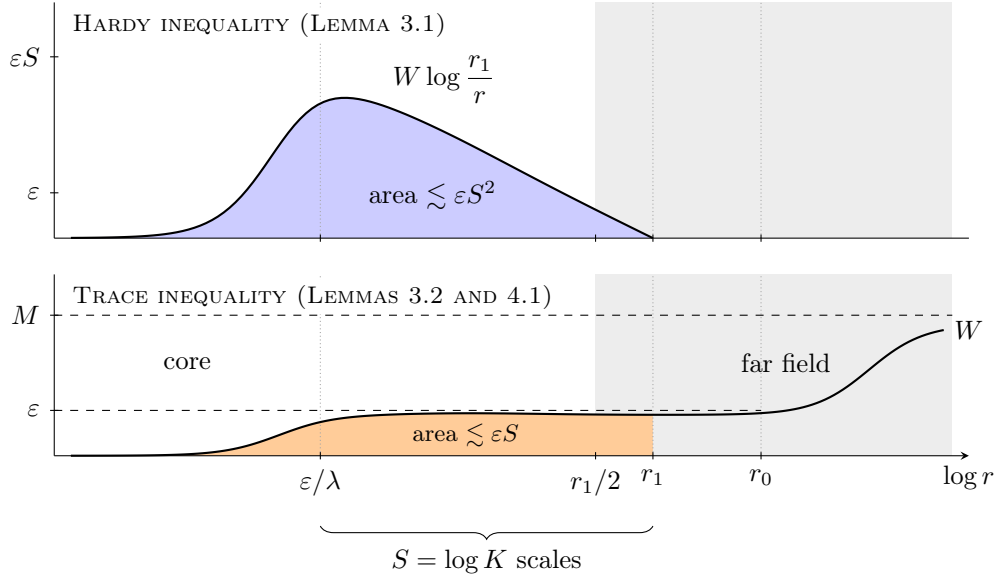

Let us point out what appears to be the main obstruction to pushing this regularity scheme further. Since $L\sim\eps S^2$, even an estimate of $T_2$ with $\Sigma \sim\eps$ would only give $k=1$, so reaching the critical case $k=0$ would require a better estimate of $T_1$. Unfortunately, the trace inequality does not seem to help, since $J$ and $\Omega$ need not vanish on the axis and the energy can control only one derivative of each. Thus, further improvements of the regularity criteria \eqref{eq:logk} below $k=1$, let alone $k=0$ where we would reach a criterion with $\Gamma$ small below an absolute threshold, appear to require substantially new ideas.

The plan of the paper is as follows. In \S\ref{sec:prelim}, we collect some preliminaries, including the refined estimate on the inner vortex core (Lemma~\ref{rad_lemma}). In \S\ref{sec:ineq}, we recall Wei's truncated Hardy inequality, which will control $T_1$, and prove the trace inequality, which is used in the estimate of $T_2$ in \S\ref{sec:T2}. Finally, in \S\ref{sec:proof}, we carry out the energy argument and prove Theorems~\ref{thm:main} and~\ref{thm:quant}.

\subsection{Some extensions of the main theorems}

We make two remarks about extensions of these ideas which we do not pursue in this work.

\begin{remark}\label{rem:avg}
Of the three places where the proof of Theorem~\ref{thm:quant} uses $|\Gamma|\le\eps$, only the first term of \eqref{eq:T2} needs a pointwise bound. In Lemmas~\ref{lem:hardy} and~\ref{lem:T2}, outside the core, $\Gamma$ enters only through the integrals $\int\Gamma^\sharp\log\frac{r_1}r\,\frac{dr}r$ and $\int(\Gamma^\sharp)^2\frac{dr}r$ over the $S$ scales below $r_1$, where we have set \begin{align*}
\Gamma^\sharp(r):=\sup_{z,t}|\Gamma|.
\end{align*}
Applying Proposition~\ref{prop:energy} to these integrals as $L$ and $\Sigma^2$, instead of bounding them by $\eps S^2$ and $\eps^2S$, one can show the following: we have regularity if both
\begin{align*}
A\coloneqq\limsup_{\ell\to\infty}\ell^{3/4}\Gamma^\sharp(e^{-\ell})<\infty\qquad\text{and}\qquad\limsup_{\ell\to\infty}\ell^{3/2}\int_\ell^{2\ell}\Gamma^\sharp(e^{-\ell'})^2\,d\ell'\leq c\min\{1,A^{-2}\}.
\end{align*}
Thus, the pointwise bound \eqref{eq:logk} is needed only with $k=3/4$, provided $\Gamma^2$ decays like $(\log r^{-1})^{-5/2}$ on average over the windows $[\ell,2\ell]$. In particular, this criterion admits $|\Gamma|\sim(\log r^{-1})^{-3/4}$ on a sparse sequence of annuli, so this is a strengthening of Theorem~\ref{thm:main}.
\end{remark}

\begin{remark}[A local version]\label{rem:local}
The same argument from our main theorems lowers the exponent in Lei and Ren's quantitative local criterion \cite[Theorem~C]{LR24} from $\frac32$ to $\frac54$. For an axisymmetric suitable weak solution in $Q(1)$, regularity in $\overline{Q(\frac12)}$ follows from
\begin{align*}
\|\Gamma\|_{L^\infty(Q(1))}\le C^{-1}\big((\min\{H,G\}+1)\log(G+2)\big)^{-5/4},    
\end{align*}
where $G$ and $H$ are the local critical quantities of \cite{LR24}. Lei and Ren localize the energy argument from \cite{W16} near the axis, with cutoff errors bounded by $e^{C\mathcal X}$, where $\mathcal X=(\min\{H,G\}+1)\log(G+2)$, so the argument closes once $S\gtrsim\mathcal X$. In the proof of Theorem~\ref{thm:quant}, we enlarge the window $S\sim\eps^{-2/3}$ to $S\sim\eps^{-4/5}$, giving the improvement to $\frac54$. Moreover, the new ingredients introduced here can be localized with errors of a comparable size.
\end{remark}

\subsection*{AI Acknowledgment}

The author used Claude 5.5 Opus to accelerate drafting the manuscript, simplifying and clarifying human-produced arguments, revising the text, and producing the TikZ figure. The model was also responsible for discovering and verifying the extra claims in Remarks \ref{rem:avg}--\ref{rem:local}. The final words are those of the author, who takes full responsibility for the content and correctness of the manuscript.

\section{Preliminaries}\label{sec:prelim}

In this section, we fix some notation and recall some standard facts about the axisymmetric Navier--Stokes equations. First, it is useful to record that
\begin{align*}
    |\nabla u|^2=|\nabla\ur|^2+|\nabla\ut|^2+|\nabla\uz|^2+(\ur/r)^2+(\ut/r)^2,
\end{align*}
which implies the pointwise bound
\begin{align}\label{pointwise_bound}
    \frac{|\ut|}r,\,\frac{|\ur|}r\leq |\nabla u|.
\end{align}
We may write the vorticity in cylindrical coordinates $\omega=\omega^re_r+\omega^\theta e_\theta+\omega^ze_z$ with
\begin{align*}
    \omega^r=-\partial_z\ut,\quad \omega^\theta=\partial_z\ur-\partial_r\uz, \quad \omega^z=\frac1r\partial_r(r\ut).
\end{align*}
Let us also record the maximum principle (see, e.g., \cite{CSTY08,CL02}): we have
\begin{align}\label{max_principle}
    |\Gamma(t,r,z)|\leq M\qquad \forall t\in[0,T_*),
\end{align}
or equivalently $|\ut|\leq M/r$.

Next, from the equations \eqref{eq:OJintro} for $\Omega$ and $J$, we have the following energy identities. These can be found, for instance, in \cite{CFZ17}.

\begin{lemma}\label{lem:energy}
For $t\in(0,T_*)$, we have
\begin{align}
\frac{d}{dt}\|\Omega\|_2^2+2\|\nabla\Omega\|_2^2+2\int_\R\Omega(t,0,z)^2dz&=T_1,\label{eq:EO}\\
\frac{d}{dt}\|J\|_2^2+2\|\nabla J\|_2^2+2\int_\R J(t,0,z)^2dz&=2T_2,\label{eq:EJ}
\end{align}
where
\begin{align}
T_1=-4\int\frac{\ut}rJ\Omega\,dx
\end{align}
and $T_2$, defined in \S\ref{sec:strategy}, can be written as
\begin{align}\label{eq:T2}
T_2=\int\ut\Big(\partial_zJ\,\partial_r\frac{\ur}r-\partial_rJ\,\partial_z\frac{\ur}r\Big)\,dx .
\end{align}
\end{lemma}

\begin{proof}
    The identities \eqref{eq:EO}--\eqref{eq:EJ} with $T_1$ and $T_2$ as in \eqref{T1T2} are routine. To see that $T_2$ can be rewritten as in \eqref{eq:T2}, note that $r(\omega^r,\omega^z)=(-\partial_z\Gamma,\partial_r\Gamma)$. Then we compute, starting from \eqref{T1T2},
    \begin{align*}
        T_2&=\int J(\omega^r\partial_r+\omega^z\partial_z)\frac{u^r}rdx\\
        &=2\pi\int\int J(-\partial_z,\partial_r)\Gamma\cdot(\partial_r,\partial_z)\frac{u^r}rdrdz\\
        &=2\pi\int\int(\partial_z,-\partial_r)J\cdot(\partial_r,\partial_z)\frac{u^r}r\Gamma\,drdz
    \end{align*}
    where, in the final equality, we use that the two differential operators are perpendicular in $(r,z)$, and that the boundary term vanishes thanks to $\Gamma$. The conclusion \eqref{eq:T2} follows because $2\pi\Gamma drdz=u^\theta dx$.
\end{proof}

We now record some bounds on $u^r/r$ which were proved in \cite{W16}; see also \cite{HLL08}.

\begin{lemma}\label{lem:f}
The quantity $u^r/r$ solves
\begin{align*}
    \Big(\Delta+\frac2r\partial_r\Big) \frac{u^r}r=\partial_z\Omega
\end{align*}
and obeys the inequalities
\begin{align*}
    \|\nabla \frac{u^r}r\|_2\leq\|\Omega\|_2,\qquad \|\nabla^2\frac{u^r}r\|_2\leq\|\partial_z\Omega\|_2,\qquad \frac1r\Big|\partial_r\frac{u^r}r\Big|\leq\Big|\nabla^2\frac{u^r}r\Big|.
\end{align*}
\end{lemma}

For the swirl component of the velocity, we define the following quantity representing the velocity scale near the axis at a fixed time,
\begin{align}\label{lambda_def}
\at(t)\coloneqq \big(\|J(t)\|_2\,\|\ut(t)/r\|_2\big)^{1/2}\leq\big(\|J(t)\|_2\,\|\nabla u(t)\|_2\big)^{1/2},
\end{align}
where the last inequality follows from \eqref{pointwise_bound}. As a comparison, Wei in \cite{W16} used the quantity $a(t)\coloneqq\sup_{r,z}\frac1r\int_0^r|\ut(t,r',z)|\,dr'$ for the same purpose and proved that $a(t)\leq\at(t)$. The following lemma strengthens this by moving the supremum in $z$ inside the integral. This is needed in the proof of Lemma~\ref{lem:T2}, where $\int_0^{r_1}\sup_z\Gamma^2\,dr/r$ must be controlled down to the axis.
\begin{lemma}\label{rad_lemma}
For $t\in(0,T_*)$ and $\rho>0$, we have
\begin{align*}
\int_0^\rho\sup_{z\in\R}|\ut(t,r,z)|\,dr\leq\rho\,\at(t)
\end{align*}
where $\at(t)$ is as in \eqref{lambda_def}.
\end{lemma}
\begin{proof}
Since $\ut(t,r,z)\to0$ as $z\to\infty$ and
\begin{align*}
\partial_z|\ut|^2=2\ut\partial_z\ut=-2r\ut J,
\end{align*}
we have
\begin{align*}
|\ut(r,z)|^2=\int_z^\infty2r\ut J\,dz'\le2r\int_\R|\ut J|(r,z')\,dz'
\end{align*}
for every $z$. Hence, by repeated applications of Cauchy--Schwarz, we conclude
\begin{align*}
\int_0^\rho\sup_z|\ut|\,dr&\leq\Big(\int_0^\rho r\,dr\Big)^{1/2}\Big(\int_0^\rho\frac{\sup_z|\ut|^2}r\,dr\Big)^{1/2}\\
&\leq\frac\rho{\sqrt2}\Big(2\iint|\ut J|\,dr\,dz\Big)^{1/2}\\
&=\rho\Big(\int\frac{|\ut|}r|J|\,dx\Big)^{1/2}\\
&\leq\rho\at .
\end{align*}
\end{proof}

\section{Hardy and trace inequalities}\label{sec:ineq}

In this section, we collect the two essential ingredients for controlling the solution near the axis.

\subsection{The truncated Hardy inequality}
We begin with a Hardy-type inequality for the weight $|u^\theta|/r=|\Gamma|/r^2$ truncated at the radius $r_1$, which will be useful for estimating $T_1$. The following statement is adapted from Lemma 2.3 in \cite{W16}.

\begin{lemma}\label{lem:hardy}
Let $t\in(0,T_*)$, $0<\eps\leq1$, and $K\geq1$, and set $L\coloneqq \eps(1+\log K)^2$. Let $r_1>0$ with $r_1\leq\eps K/\at(t)$ and
\begin{align*}
|\Gamma(t)|\leq\eps\text{ on }\{r\leq r_1\}.
\end{align*}
Then for every smooth axisymmetric scalar function $h\in H^1 $, we have
\begin{align*}
\int\frac{|\ut|}r|h|^2dx\leq L\,\|\partial_rh\|_2^2+O\Big(\frac{L+M}{r_1^2}\int_{r\geq r_1/2}|h|^2dx\Big).
\end{align*}
\end{lemma}
\begin{proof}
We first consider the case that $h=0$ for $r\geq r_1$. For $0<r'<r_1$, Cauchy--Schwarz gives
\begin{align*}
|h(r',z)|^2=\Big|\int_{r'}^{r_1}\partial_rh\,dr\Big|^2\leq\int_{r'}^{r_1}r|\partial_rh|^2dr\log\frac{r_1}{r'} .
\end{align*}
Defining
\begin{align*}
v(r,z)\coloneqq \int_0^r|\ut(r'',z)|\,dr''
\end{align*}
and applying Fubini's theorem followed by integration by parts in $r'$, we find
\begin{align*}
\int_0^{r_1}|\ut(r',z)||h(r',z)|^2dr'&\leq\int_0^\infty r|\partial_rh|^2dr\cdot\int_0^{r_1}|\ut(r',z)|\log\frac{r_1}{r'}dr'\\
&=\int_0^\infty r|\partial_rh|^2dr\cdot\int_0^{r_1}v(r,z)\frac{dr}r .
\end{align*}
Next, observe that $v(r,z)\leq r\at$ for all $r$ by Lemma~\ref{rad_lemma}. For $\eps/\at\leq r\leq r_1$, we have $|\ut|=|\Gamma|/r\leq\eps/r$, and hence, by the fundamental theorem of calculus,
\begin{align*}
v(r,z)\leq v(\eps/\at,z)+\int_{\eps/\at}^r|u^\theta(r',z)|dr'\leq \eps+\eps\log\frac{r\at}\eps.
\end{align*}
If $r_1\leq\eps/\at$, then we estimate
\begin{align*}
\int_0^{r_1}v\,\frac{dr}r\leq\at r_1\leq\eps\leq L.
\end{align*}
Otherwise,
\begin{align*}
\int_0^{r_1}v\frac{dr}r\leq\int_0^{\eps/\at}\at\,dr+\int_{\eps/\at}^{\eps K/\at}\eps\Big(1+\log\frac{r\at}\eps\Big)\frac{dr}r=\eps+\eps\int_0^{\log K}(1+s)\,ds\leq L .
\end{align*}
Integrating in $z$ gives
\begin{align*}
\int\frac{|\ut|}r|h|^2dx=\iint|\ut||h|^2dr\,dz\leq L\|\partial_rh\|_2^2 .
\end{align*}

Now consider general $h$. We let $\varphi:[0,\infty)\to[0,1]$ smooth and decreasing with $\varphi=1$ on $[0,\frac12]$ and $\varphi=0$ on $[1,\infty)$, then define $\varphi_1(r)=\varphi(r/r_1)$. Applying the above to $\varphi_1h$, and using the swirl a priori bound on the rest, we obtain
\begin{align*}
\int\frac{|\ut|}r|h|^2dx\leq L\|\partial_r(\varphi_1h)\|_2^2+\frac{4M}{r_1^2}\int_{r\geq r_1/2}|h|^2dx .
\end{align*}
Next, we expand
\begin{align*}
\|\partial_r(\varphi_1h)\|_2^2=\int\big(\varphi_1^2|\partial_rh|^2+(\varphi_1')^2|h|^2\big)dx+\iint\varphi_1\varphi_1'\,\partial_r|h|^2\,r\,dr\,dz .
\end{align*}
We integrate by parts in the last term and use that $|\partial_r(r\varphi_1\varphi_1')|\lesssim r_1^{-1}$ on $[r_1/2,r_1]$ to bound
\begin{align*}
\Big|\iint|h|^2\partial_r(r\varphi_1\varphi_1')\,dr\,dz\Big|\lesssim\frac1{r_1^2}\int_{r\geq r_1/2}|h|^2dx.
\end{align*}
Thus
\begin{align*}
\|\partial_r(\varphi_1h)\|_2^2\leq\|\partial_rh\|_2^2+O(r_1^{-2}\int_{r\geq r_1/2}|h|^2dx),
\end{align*}
and the lemma follows.
\end{proof}

\subsection{Trace estimate near the axis}

Next we prove a trace inequality for $\partial_z(\ur/r)$ on the cylinders $\{r=\text{const}\}$, uniform in $r$, which we use in place of a Hardy inequality to estimate $T_2$.

\begin{lemma}\label{lem:trace}
Let $g\in H^2(\R^3)$ be axisymmetric  and write $x=(x',z)$ with $x'\in\R^2$. The map $x'\mapsto\partial_zg(x',\cdot)\in L^2(\R)$ has a continuous
representative, and moreover it satisfies
\begin{align*}
\sup_{r>0}\int_\R|\partial_zg(r,z)|^2dz\lesssim\|\Delta g\|_2^2 .
\end{align*}
\end{lemma}

\begin{proof}
Let $\hat g(\xi,k)$ be the full Fourier transform of $g$, with $\xi\in\R^2$ corresponding to $x'$ and $k\in\R$ to $z$. For fixed $x'\in\R^2$, define the Fourier transform of $\partial_zg(x',\cdot)$ in $z$,
\begin{align*}
G(x',k)\coloneqq\frac1{(2\pi)^2}\int_{\R^2}e^{ix'\cdot\xi}\,ik\,\hat g(\xi,k)\,d\xi .
\end{align*}
By Cauchy--Schwarz, we estimate
\begin{align*}
|G(x',k)|^2\lesssim\int_{\R^2}(|\xi|^2+k^2)^2|\hat g(\xi,k)|^2\,d\xi=\int_{\R^2}|\widehat{\Delta g}(\xi,k)|^2\,d\xi,
\end{align*}
having used that $\int_{\mathbb R^2}k^2/(|\xi|^2+k^2)^2d\xi=O(1)$. By Plancherel's theorem, first in $z$ and then in all of $\R^3$, we conclude that
\begin{align*}
\int_\R|\partial_zg(x',z)|^2dz=\frac1{2\pi}\int_\R|G(x',k)|^2dk\lesssim\int_{\R^3}|\widehat{\Delta g}|^2\,d\xi\,dk\sim\int_{\R^3}|\Delta g|^2\,dx'\,dz .
\end{align*}
\end{proof}

\section{The stretching term}\label{sec:T2}

We now estimate the stretching term $T_2$ defined in \eqref{eq:T2} under the hypotheses of Lemma~\ref{lem:hardy}, as outlined in \S\ref{sec:strategy}.

\begin{lemma}\label{lem:T2}
Let $t$, $\eps$, $K$, and $r_1$ be as in Lemma~\ref{lem:hardy}. Then
\begin{align*}
|T_2|\lesssim\eps\big((1+\log K)^{1/2}\|\partial_rJ\|_2+\|\partial_zJ\|_2\big)\|\partial_z\Omega\|_2+\frac M{r_1}\|\nabla J\|_{L^2(r\geq r_1)}\|\nabla(\ur/r)\|_{L^2(r\geq r_1)} .
\end{align*}
\end{lemma}
\begin{proof}
We decompose the integral in \eqref{eq:T2} at radius $r=r_1$. On $\{r\geq r_1\}$, we have $|\ut|\leq M/r_1$ by \eqref{max_principle}. Since the
integrand is at most $|\ut||\nabla J||\nabla(\ur/r)|$, this region contributes the last term in the claimed inequality. On $\{r<r_1\}$, we have $|\Gamma|\leq\eps$, and we treat the two terms in \eqref{eq:T2} separately.

First, for $\ut\,\partial_zJ\,\partial_r(\ur/r)$, Lemma~\ref{lem:f} gives
\begin{align*}
\Big|\ut\,\partial_r\frac{\ur}r\Big|=|\Gamma|\,\Big|\frac1r\partial_r\frac{\ur}r\Big|\leq\eps\Big|\nabla^2\frac{\ur}r\Big|,
\end{align*}
so, applying Lemma~\ref{lem:f} again, we have the upper bound
\begin{align*}
\eps\|\partial_zJ\|_2\|\nabla^2(\ur/r)\|_2\leq\eps\|\partial_zJ\|_2\|\partial_z\Omega\|_2 .
\end{align*}

Consider next the second term, $\ut\,\partial_rJ\,\partial_z(\ur/r)$. Let $W(r)\coloneqq \sup_z|\Gamma(t,r,z)|$. Since $|\ut|\,r=|\Gamma|$, the Cauchy--Schwarz inequality in $z$ and then in $r$ gives
\begin{align*}
\int_{r<r_1}|\ut||\partial_rJ||\partial_z(\ur/r)|\,dx&\leq\int_0^{r_1}W(r)\,\|\partial_rJ(r,\cdot)\|_{L^2_z}\|\partial_z(\ur/r)(r,\cdot)\|_{L^2_z}dr\\
&\leq\sup_{r>0}\|\partial_z(\ur/r)(r,\cdot)\|_{L^2_z}\Big(\int_0^{r_1}W^2\frac{dr}r\Big)^{1/2}\Big(\int_0^\infty\|\partial_rJ(r,\cdot)\|^2_{L^2_z}r\,dr\Big)^{1/2}\\
&\lesssim\|\Delta(\ur/r)\|_2\Big(\int_0^{r_1}W^2\frac{dr}r\Big)^{1/2}\|\partial_rJ\|_2,
\end{align*}
having used Lemma~\ref{lem:trace}. After another application of Lemma~\ref{lem:f} to bound $\|\Delta(u^r/r)\|_2$, it remains to show that
\begin{align*}
\int_0^{r_1}W^2\frac{dr}r\leq\eps^2(1+\log K).
\end{align*}
Let $\rho\coloneqq \min\{\eps/\at,r_1\}$. On $(0,\rho)$, we use $W\leq\eps$
and $W(r)/r=\sup_z|\ut(r,\cdot)|$ to find
\begin{align*}
\int_0^\rho W^2\frac{dr}r\leq\eps\int_0^\rho\sup_z|\ut|\,dr\leq\eps\rho\,\at\leq\eps^2 ,
\end{align*}
by Lemma~\ref{rad_lemma}. If $\rho<r_1$, then on $(\rho,r_1)$ we
have $W\leq\eps$ and $\log(r_1/\rho)=\log(r_1\at/\eps)\leq\log K$, so
\begin{align*}
\int_\rho^{r_1}W^2\frac{dr}r\leq\eps^2\log K.
\end{align*}
Combining this once again with Lemma~\ref{lem:f}, we conclude that this term is bounded by
\begin{align*}
\eps(1+\log K)^{1/2}\|\partial_rJ\|_2\|\partial_z\Omega\|_2
\end{align*}
up to a constant.
\end{proof}

\section{Main regularity proofs}\label{sec:proof}

\subsection{The energy argument}
We now carry out the energy argument of \cite{W16}, with the Hardy constant $L$ and the coefficient $\Sigma$ of the stretching term kept as
parameters.

\begin{proposition}\label{prop:energy}
For any $C_1\geq1$, there exist $c_0,C_0>0$ with $c_0$ small, with the following property. Let $u$ be as
in Theorem~\ref{thm:main}, and let $\eps\in(0,\frac12]$ and $r_0,K>0$, $L\geq\eps$, and $\Sigma\ge0$ be
such that 
\begin{align}\label{eq:win}
(\Sigma^2+\eps^2)L^2\leq c_0 .
\end{align}
For $t\in(0,T_*)$, define
\begin{align*}
r_1(t)\coloneqq \min\{\eps K/\at(t),\,r_0\},
\end{align*}
and suppose that for every $t\in(0,T_*)$, we have
\begin{align}
\int\frac{|\ut|}r|h|^2dx&\leq L\,\|\partial_rh\|_2^2+C_1\frac{L+M}{r_1^2}\int_{r\geq r_1/2}|h|^2dx\label{eq:H}
\end{align}
for $h\in\{J(t),\Omega(t)\}$, as well as
\begin{align}
|T_2|&\leq C_1\big(\Sigma\|\partial_rJ\|_2+\eps\|\partial_zJ\|_2\big)\|\partial_z\Omega\|_2+C_1\frac M{r_1}\|\nabla J\|_{L^2(r\geq r_1)}\|\nabla(\ur/r)\|_{L^2(r\geq r_1)} .\label{eq:T}
\end{align}
If $K\geq C_0\eps^{-5}X$, with $X=\max\{M_0^{1/4},r_0^{-1/2}M_1\}$ as in \S\ref{sec:main}, then $T_*=\infty$.
\end{proposition}

\begin{proof}
Consider first an energy inequality for $J$. Using \eqref{eq:EJ}, \eqref{eq:T}, and Young's inequality, and dropping the axis term, we obtain
\begin{align}\label{eq:s1}
\frac{d}{dt}\|J\|_2^2+\|\nabla J\|_2^2\lesssim (\Sigma^2+\eps^2)\,\|\partial_z\Omega\|_2^2+\frac{M^2}{r_1^2}\|\nabla(\ur/r)\|^2_{L^2(r\geq r_1)} 
\end{align}
where, here and elsewhere, the implicit constants may depend on $C_1$.

Next, we turn to an energy estimate for $\Omega$. By \eqref{eq:EO}, Cauchy--Schwarz, and
Young's inequality, we have
\begin{align*}
4\Big|\int\frac{\ut}rJ\Omega\,dx\Big|\leq\frac2L\int\frac{|\ut|}r\Omega^2dx+2L\int\frac{|\ut|}rJ^2dx .
\end{align*}
Applying \eqref{eq:H} with $h=\Omega$ and $h=J$, and dropping the axis term, we
find
\begin{align*}
\frac{d}{dt}\|\Omega\|_2^2+2\|\nabla\Omega\|_2^2\leq 2\|\partial_r\Omega\|_2^2+2L^2\|\partial_rJ\|_2^2+2C_1\frac{L+M}{r_1^2}\int_{\{r\geq r_1/2\}}\Big(L^{-1}|\Omega|^2+L|J|^2\Big)dx.
\end{align*}
Exploiting the exact cancellation of the $2\|\partial_r\Omega\|_2^2$ on either side, this implies
\begin{align}\label{eq:s2}
    \frac{d}{dt}\|\Omega\|_2^2+2\|\partial_z\Omega\|_2^2\leq 2L^2\|\partial_rJ\|_2^2+O\Big(\frac{L+M}{r_1^2}\Big)\int_{\{r\geq r_1/2\}}\Big(L^{-1}|\Omega|^2+L|J|^2\Big)dx.
\end{align}
With $\mu\coloneqq \frac1{4L^2}$, we now set
\begin{align*}
    A\coloneqq \|J\|_2^2+\mu\|\Omega\|_2^2
\end{align*}
and add \eqref{eq:s1} to \eqref{eq:s2}, the latter weighted by $\mu$. Using $2\mu L^2=\frac12$ and splitting $\|\nabla J\|_2^2$ into its components, we arrive at
\begin{align*}
\frac{d}{dt}A+\frac12\|\partial_rJ\|_2^2+\|\partial_zJ\|_2^2+\big(2\mu-O(\Sigma^2+\eps^2)\big)\|\partial_z\Omega\|_2^2\leq\mathcal F,
\end{align*}
where $\mathcal F$ collects the far-field terms with integration over $\{r\geq r_1/2\}$. The coefficient of $\|\partial_z\Omega\|_2^2$ is clearly non-negative by \eqref{eq:win} and the choice of $c_0$. We estimate
\begin{align*}
    \mathcal F&\lesssim \frac{M^2}{r_1^2}\|\nabla(\ur/r)\|^2_{L^2(r\geq r_1)} +\mu\frac{L+M}{r_1^2}\int_{\{r\geq r_1/2\}}\Big(L^{-1}|\Omega|^2+L|J|^2\Big)dx\\
    &\lesssim (M^2+\mu(1+\frac ML+L^2))\mathcal R
\end{align*}
where $\mathcal R$ is the truncated energy
\begin{align*}
\mathcal R(t)\coloneqq \frac1{r_1^2}\int_{r\geq r_1/2}\Big(\Big|\nabla\frac\ur r\Big|^2+\Omega^2+J^2\Big)dx .
\end{align*}
By \eqref{eq:win}, we have $\eps L\leq c_0^{1/2}\leq\frac14$, so
\begin{align}\label{eq:Lrange}
\eps\leq L\leq\frac1{4\eps}.
\end{align}
Thus,
\begin{align*}
\mathcal F&\lesssim (1+M)^2\eps^{-3}\mathcal R
\end{align*}
which implies
\begin{align}\label{eq:s3}
\frac{d}{dt}A\lesssim(1+M)^2\eps^{-3}\,\mathcal R(t).
\end{align}

It remains to control the far-field terms. By Lemma~\ref{lem:f} and \eqref{eq:Lrange},
\begin{align*}
\int(|\nabla(\ur/r)|^2+\Omega^2+J^2)dx\lesssim\|\Omega\|_2^2+\|J\|_2^2\lesssim(1+L^2)A\lesssim\eps^{-2}A .
\end{align*}
On $\{r\geq r_1/2\}$, the identity $\nabla(\ur/r)=\nabla\ur/r-(\ur/r^2)e_r$ gives
\begin{align*}
|\nabla(\ur/r)|^2\lesssim r_1^{-2}\big(|\nabla\ur|^2+(\ur/r)^2\big),\qquad\Omega^2+J^2\leq r^{-2}|\omega|^2\lesssim r_1^{-2}|\nabla u|^2 .
\end{align*}
Combining these bounds with \eqref{pointwise_bound}, we have
\begin{align*}
\mathcal R(t)\lesssim\frac1{r_1^2}\min\Big\{\eps^{-2}A,\ \frac1{r_1^2}\|\nabla u\|_2^2\Big\}.
\end{align*}
Consider first the case $r_1=\eps K/\at\leq r_0$. Since
$\at^2=\|J\|_2\|\ut/r\|_2\leq A^{1/2}\|\nabla u\|_2$, we have
\begin{align*}
r_1^{-2}=\at^2(\eps K)^{-2}\leq A^{1/2}\|\nabla u\|_2(\eps K)^{-2} .
\end{align*}
Then, using $\min\{x,y\}\leq x^{2/3}y^{1/3}$, we obtain
\begin{align*}
\mathcal R&\lesssim\frac{A^{1/2}\|\nabla u\|_2}{(\eps K)^2}\min\Big\{\eps^{-2}A,\frac{A^{1/2}\|\nabla u\|_2^3}{(\eps K)^2}\Big\}
\\
&=\frac{A\|\nabla u\|_2}{(\eps K)^2}\min\Big\{\eps^{-2}A^{1/2},\frac{\|\nabla u\|_2^3}{(\eps K)^2}\Big\}\\
&\leq\frac{A\|\nabla u\|_2}{(\eps K)^2}\big(\eps^{-2}A^{1/2}\big)^{2/3}\Big(\frac{\|\nabla u\|_2^3}{(\eps K)^2}\Big)^{1/3}
\\
&=\eps^{-4/3}\frac{A^{4/3}\|\nabla u\|_2^2}{(\eps K)^{8/3}},
\end{align*}
using that the upper bound for $\mathcal R(t)$ is increasing in $r_1^{-2}$. In the remaining case $r_1=r_0$, we simply have $\mathcal R\lesssim r_0^{-4}\|\nabla u\|_2^2$. Inserting these bounds into
\eqref{eq:s3}, we find that
\begin{align}\label{eq:ode}
\frac{d}{dt}A\leq C_3(1+M)^2\eps^{-13/3}\|\nabla u\|_2^2\,G(A),\qquad G(y)\coloneqq \max\Big\{\frac{y^{4/3}}{(\eps K)^{8/3}},\frac1{r_0^4}\Big\},
\end{align}
with a constant $C_3$ depending only on $C_1$.

We now conclude by a standard continuity argument. Let
\begin{align*}
y_*\coloneqq (\eps K)^2r_0^{-3},\qquad A_*\coloneqq \max\{A(0),y_*\}.
\end{align*}
Then $y_*$ is the point where the two expressions in $G$ agree, so that
$G(y)=y^{4/3}(\eps K)^{-8/3}$ for $y\geq y_*$. Let $I$ be the set of $t\in[0,T_*)$ with $A\le2A_*$ on $[0,t]$.
Clearly $I$ is a closed interval containing $0$, since $A$ is continuous. For $t\in I$, integrating \eqref{eq:ode} over $(s,t)$, letting
$s\to0$, and using the energy equality and the fact that $G$ is non-decreasing, we obtain
\begin{align*}
A(t)&\leq A(0)+C_3(1+M)^2\eps^{-13/3}G(2A_*)\int_0^t\|\nabla u\|_2^2ds\\&\leq A_*+\frac{C_3}2\eps^{-13/3}M_1^2\,(2A_*)^{4/3}(\eps K)^{-8/3}.
\end{align*}
Since $A_*^{1/3}\leq A(0)^{1/3}+y_*^{1/3}$ and $y_*^{1/3}=(\eps K)^{2/3}r_0^{-1}$, the last term is less than $A_*$ as soon as
\begin{align*}
(\eps K)^{8/3}\geq C\,\eps^{-13/3}M_1^2A(0)^{1/3}\quad\text{and}\quad(\eps K)^2\geq C\,\eps^{-13/3}M_1^2r_0^{-1}
\end{align*}
for a suitable constant $C$ depending only on $C_1$. Since $\mu^{1/2}=\frac1{2L}\leq\frac1{2\eps}$,
\begin{align*}
A(0)^{1/2}\leq\|J_0\|_2+\mu^{1/2}\|\Omega_0\|_2\leq\eps^{-1}(\|J_0\|_2+\|\Omega_0\|_2) ,
\end{align*}
and therefore $M_1^2A(0)^{1/3}\leq\eps^{-2/3}M_0^{2/3}$. Hence both conditions hold provided
\begin{align*}
K\geq C'\eps^{-23/8}M_0^{1/4}\quad\text{and}\quad K\geq C'\eps^{-19/6}r_0^{-1/2}M_1
\end{align*}
for a constant $C'$ depending only on $C_1$. Setting $C_0= C'+1$, both conditions follow from the hypothesis $K\geq C_0\eps^{-5}X$. Thus $A(t)<2A_*$ for $t\in I$, so $I$ is also open
in $[0,T_*)$. We conclude that $I=[0,T_*)$, i.e., $A\le2A_*$ on $[0,T_*)$.

Finally, since $\|\Omega(t)\|_2^2\leq A(t)/\mu$ is bounded on $(0,T_*)$, we conclude that $T_*=\infty$ by a standard regularity criterion, see, for instance, \cite[Lemma~3.1]{CFZ17} or \cite{W16}.
\end{proof}

\subsection{Proof of theorems}\label{sec:proofs}

The main theorems are now a routine application of the energy inequalities we have established.

\begin{proof}[Proof of Theorem~\ref{thm:quant}]
Let $C_1$ be large compared to the implicit constants in Lemmas~\ref{lem:hardy} and~\ref{lem:T2}, and let $c_0'$ and $C_0$ be the corresponding constants of Proposition~\ref{prop:energy}. We set $c_0\coloneqq c_0'/2$ and define
\begin{align*}
K\coloneqq e^S,\qquad L\coloneqq \eps(1+S)^2,\qquad\Sigma\coloneqq \eps(1+S)^{1/2}.
\end{align*}
Then $L\geq\eps$ and, by \eqref{ass1},
\begin{align*}
(\Sigma^2+\eps^2)L^2\le2\eps^4(1+S)^5\leq 2c_0=c_0',
\end{align*}
which gives the hypothesis \eqref{eq:win}; moreover, $\eps\leq c_0^{1/4}\leq\frac12$. Let
$r_1(t)\coloneqq \min\{\eps K/\at(t),r_0\}$. Since $r_1\leq\eps K/\at$ and $|\Gamma(t)|\leq\eps$ on $\{r\leq r_1\}\subset\{r\leq r_0\}$,
Lemmas~\ref{lem:hardy} and~\ref{lem:T2} apply at every $t\in(0,T_*)$ and give \eqref{eq:H} and \eqref{eq:T}. Finally, \eqref{ass2} is
precisely $K\geq C_0\eps^{-5}X$, so Proposition~\ref{prop:energy} gives $T_*=\infty$.
\end{proof}

\begin{proof}[Proof of Theorem~\ref{thm:main}]
Let $c_0$ and $C_0$ be the constants of Theorem~\ref{thm:quant}, and set $c_*\coloneqq c_0^{1/4}$. For $r_0\in(0,\delta_0]$, let
\begin{align*}
\eps(r_0)\coloneqq c_*(\log r_0^{-1})^{-5/4},
\end{align*}
so that
\begin{align*}
S_*(\eps(r_0))=c_0^{1/5}\eps(r_0)^{-4/5}-1=\log r_0^{-1}-1.
\end{align*}
Since $\eps(r_0)\to0$ as $r_0\to0$, we may take $r_0$ small enough
that $\eps(r_0)\leq c_0^{1/4}$. Since the function $r\mapsto(\log r^{-1})^{-5/4}$ is increasing on $(0,1)$, the hypothesis gives $|\Gamma|\leq\eps(r_0)$ on $\{r\leq r_0\}\times(0,T_*)$. With $S\coloneqq S_*(\eps(r_0))$, \eqref{ass1} holds by definition, while \eqref{ass2} requires
\begin{align*}
S_*(\eps(r_0))&\geq\log C_0+5\log\frac1{\eps(r_0)}+\max\Big\{\frac14\log M_0,\ \frac12\log r_0^{-1}+\log M_1\Big\}\\
&=\frac12\log r_0^{-1}+O(\log\log r_0^{-1}),
\end{align*}
because $\log\frac1{\eps(r_0)}=\frac54\log\log r_0^{-1}+O(1)$. Since $S_*(\eps(r_0))=\log r_0^{-1}-1$, this holds for all sufficiently small $r_0$, and thus Theorem~\ref{thm:quant} applies.
\end{proof}

\bibliographystyle{abbrv}
\bibliography{refs}

\end{document}